\documentclass[11pt,reqno]{amsart}
\usepackage[body={7in,9in},centering]{geometry}
\usepackage{amsmath,amssymb,amsthm,mathrsfs}
\usepackage{color,xcolor}
\definecolor{cobalt}{RGB}{61,99,181}
\usepackage[colorlinks,citecolor=cobalt,linkcolor=cobalt]{hyperref}
\usepackage{float}
\usepackage{exscale}
\usepackage{relsize}
\usepackage{graphicx}
\usepackage{tikz}
\newtheorem{thm}{Theorem}[section]
\newtheorem{defi}[thm]{Definition}

\newtheorem{lem}[thm]{Lemma}

\numberwithin{equation}{section}

\date{\today}

\makeatletter

\newcommand{\Rmnum}[1]{\expandafter\@slowromancap\romannumeral #1@}
\makeatother

\begin{document}

	\title[Height estimate in$\rho R\times M^2$ ]{Height estimate for special weingarten surfaces of elliptic type in $ R\times_{\rho} M^2$}
	\author[Xia Liu]{Xia LIu}
	\address{
		\textsuperscript{1}
		Chongqing University}\
	\email{XIA1246249827@163.com}
	

	\begin{abstract}
		In this paper, we provide a vertical height estimate for compact special Weingarten surface of elliptic type in  $ R\times_{\rho} M^2$,i.e. surface whose mean curvature $H$ and extrinsic Gauss curvature $K_{e}$ satisfy $H=f(H^2-K_{e})$ with $4x(f'(x))^2<1$,for all $x\in [0,+\infty )$.
	\end{abstract} \maketitle
	
	\section{Introduction}\label{S1}
	
	In this paper we will consider special Weingarten surfaces of elliptic type in $ R\times_{\rho} M^2$.If $H$ and $K_{e}$ denote the mean curvature and the extrinsic Guass curvature of a surface $\Sigma $ respectively,then  $\Sigma $ is called a special Weingarten surface if the following identity holds:$$H=f(H^2-K_{e}) $$,
	with $f\in \mathcal{C} ^{0} ([0,+\infty ))$.Furthermore if $f\in \mathcal{C} ^{1} ([0,+\infty ))$ and $4x(f'(x))^2<1,\forall x\in [0,+\infty )$,then $f$ is said to be elliptic and $\Sigma $ is said to be a special Weingarten surface of elliptic type,henceforth called a SWET surface.\
	
	The problem of mean curvature has been a hot topic in geometric directions, especially for minimal surfaces, i.e., where the mean curvature is equal to 0. In the recent years, the study of constant mean curvature problems has gradually become popular,such as the study of height estimate for a surface  .\
	
	The height estimate was first obtained by Heinz in 1969\cite{Al1}, he showed that a compact graph with positive constant mean curvature (H-surface) in a three-dimensional Euclidean space, which boundary is in the plane, will have a maximum value of $\frac{1}{H}$  .Addtionally,Serrin also get similar conclusion in \cite{Al2}. \
	
	In 1992,Korevaar,N,et al.  \cite{Al3}  proved the height estimate of surfaces in Hyperbolic space.And In 1993, H.Rosenberg \cite{Al4} proved that a compact hypersurface with positive constant higher order mean curvature ,embedding in n-dimensional Euclidean space,which boundary value is equal to 0 ,has a maximum height of $2(H_{r} +1)^{\frac{1}{r+1} } $. The auther popularized the research of Heinz \cite{Al1} and Korevaar,N,et al. \cite{Al3} by generalizing graph to embedded hypersurfaces and constant mean curvature to constant higher order mean curvature.\
	
	Since the height estimate of surfaces have results in Euclidean space, does it have a similar conclusion in manifolds? In 2003, Hoffman,H et al. \cite{Al5} proved that in the product space $M \times  [a,\infty 
	)$,the compact and embedded $H$-surface, whose boundary is in $M \times {a}$, has a maximum of its height if the Gaussian curvature of $M$ satisfies $K_{M}\ge 2\tau,\tau <0,$ and $H^{2}>|t|$. In 2008, Aledo,J.A et al. \cite{Al6} also proved a height estimate in a two-dimensional product space, and he also inscribed an example where the equation holds.The height estimate proved by Aledo,J.A et al. \cite{Al6} is based on a generalization of Hoffman,H et al. \cite{Al5} in 2003, and the conclusion of the height estimate drawn by Hoffman ,H et al. does not apply to all two-dimensional product spaces, i.e., when the product space is $H ^2 \times  R$, this conclusion does not hold. Therefore, the conclusion of Aledo et al. is the best result in two-dimensional product spaces, i.e., it is valid for any two-dimensional product space.\
	
	The conclusion of the height estimate proved by Alias,L.J et al. in 2007 \cite{Al7} is a direct generalization of the conclusion of Hoffman et al. in 2003 \cite{Al5} to higher dimensional product space. In 2005, Cheng,X et al. \cite{Al8} generalized the conclusions of the two-dimensional product to the high-dimensional product space, and  generalized the constant mean curvature to the constant higher-order mean curvature . The difference with the conclusions proved by Alias et al. in 2007 \cite{Al7} is that  Cheng,X et al. worked in the case of constant higher-order mean curvature while Alias,L.J et al. worked in the case of constant mean curvature. And furthermore,  Cheng,X et al. gave conditional constraints on the Gaussian curvature of Riemannian manifolds while Alias,L.J et al. gave conditional constraints on the Ricci curvature of Riemannian manifolds.\
	
	Warped product space is a more specialized type of product space. Since there are conclusions about height estimates in general product spaces, does the same hold true for Warped products?\
	
	In 2007 Alias,L.J et al.  \cite{Al7} proved a height estimate for hypersurfaces with constant mean curvature e in $M=p^{n}\times R $ . In 2013  Alias,L.J et al.  \cite{Al9}combined this with a height estimate for Warped products.In 2015 Impera,I et al \cite{Al10} proved a sharp height estimate for compact hypersurfaces with constant k-mean curvature in warped product spaces ,which generalized the conclusions of Alias et al. 2007 to constant higher order mean curvature.In 2018 de lima,E.L et al. \cite{Al11} proved height estimates for a special class of hypersurfaces in Riemannian  Warped products, namely generalized linear Weingarten hypersurfaces. \
	
	For the conclusions mentioned above, whether in Euclidean space, in product space, or in warped product space, the conditions of the theorem either require that it be graph or that the angle function does not change sign. If the conditions are weakened by not requiring  graph or that the angle function does not change sign, must the corresponding height estimates hold?
	
	In 1996, Montiel,S et al. \cite{Al12} proved height estimates for compact surfaces with constant mean curvature in Euclidean space.In 2012, Rosenberg,H et al. \cite{Al13} proved height estimates for compact surfaces with constant mean curvature in product space. The difference with 1996 is that there are more restrictions on the external space $ M$. Here, not only is $M$ required to be a Hardmand surface, but also it is required to be a complete simple connected Riemannian  surface with constant Gaussian curvature $K \le-t,t$.
	
	If we consider a more general case, instead of requiring a surface with constant mean curvature, we require a SWET surface, i.e., a mean curvature and a Gaussian curvature satisfying the functional relation.\
	
	In 1994,H, Rosenberg et al. \cite{Al14} proved the height estimation of the SWET surface in Euclidean space, which differs from the previous study in that it uses a new linear operator to perform the calculation, which leads to a different conclusion. In 2014, Morabito et al. \cite{Al15} proved the height estimate of the SWET surface in product  space, which is a generalization of the 1994 conclusion, directly extending the Euclidean space to the product space.
	Then, the question arises whether the similar conclusion is shared in the warped product space? This is the main research of this paper. \
	
	The framework structure of this paper is divided into three main sections as follows:
	In the section 1,we mainly introduces the research background and the structure of the paper;In the section 2 ,we give the basic definitions and lemmas;In the section 3, we prove the main theorem about the height estimation of SWET surface in warped product space,the result is as follows: \
	
	\begin{thm}
		Let $\Sigma$ be a compact SWET surface,which  embedded in   $ R \times_{\rho} M^{2}$,and $\rho$ is a positive increasing function,$\partial \Sigma \subset M^{2} \times \left \{ t_{0}  \right \},t_{0} \in \mathbb{R}$.Suppose the Guassian curvature of $M$ is $K_{M}$, $K_{M}\ge Sup_{R} ((\rho^{'})^{2}-\rho^{''}\rho)$,$ H_{k} \ge Sup_{(-\infty ,t_{0} ]} \mathcal{H} _{k},k=1,2. $ if $f^{'}> 0,1-2ff^{'}>0,f^{2} +x(1-4ff^{'})>0,f-2f^{'}x>0.$then,
		
		\[h\le t_{0}+\frac{\rho(maxh)}{c\rho(t_{0})},c=\min_{\Sigma } (\frac{ \rho^{'}}{\rho} +\frac{f^{2} +x(1-4ff^{'}}{f-2f^{'}x} ).
		\]
	\end{thm}
	
	\section{Preliminary and Lemmas}\label{S2}
	Before the main therom is proved,we need to state some definitions and lemmas.More details can be found in the references \cite{Al11,Al15,Al16}.\
	
	\begin{defi}
		If $(M_{1}, g_{1})$ and $(M_{2}, g_{2})$ are Riemannian manifolds,the product manifold $M_{1}\times M_{2}$ has a natural Riemannian metric $g=g_{1} \oplus g_{2}$,called the product metric,defined by $$g_{(p_{1} ,p_{2})}((v_{1} ,v_{2}),(w_{1} ,w_{2}))=g_{2} |_{p_{1}} (v_{1},w_{1})+g_{1} |_{p_{2}} (v_{2},w_{2}),$$ where $(v_{1} ,v_{2})$ and $(w_{1} ,w_{2})$ are elements of $T_{p_{1}}M_{1}\oplus T_{p_{2}}M_{2}$,which is naturally identified with $T_{(p_{1} ,p_{2})} (M_{1}\times M_{2})$.\ Smooth local coordinates $x^{1} ,...,x^{n}$ for $M_{1}$ and   $(x^{n+1} ,...,x^{n+m})$ for $M_{2}$ give coordinates $(x^{1} ,...,x^{n+m})$ for $(M_{1}\times M_{2})$.\ In terms of these coordinates,the product metric has the local expression $g=g_{ij} dx^{i} dx^{j} (g_{ij} )$,\\where $(g_{ij})$ is the block diagonal matric
		$$(g_{ij})=\begin{pmatrix}
			(g_{1})_{ab}   & 0 \\
			0 & (g_{2})_{cd} 
		\end{pmatrix};$$
		here the indices $a,b$ run from $1$ to $n$, and $c,d$ run from $n+1$ to $n+m$.Product metrics on products of three or more Riemannian manifolds are defined similarly.
		
	\end{defi}
	Here is an important generalization of product metrics.
	\begin{defi}
		Suppose $(M_{1},g_{1})$ and $(M_{1},g_{1})$ are two Riemannian manifolds,and $f:M_{1}\longrightarrow R^{+} $ is a strictly positive smooth function.The warped product $M_{1} \times_{f}M_{2} $ is the product manifold $M_{1} \times M_{2} $ endowed with the Riemannian metric $g=g_{1} \oplus f^{2}g_{2}$,defined by $$g_{(p_{1} ,p_{2})}((v_{1} ,v_{2}),(w_{1} ,w_{2}))=g_{1} |_{p_{1}} (v_{1},w_{1})+f(p_{1})^{2} (g_{2} |_{p_{2}} (v_{2},w_{2}) $$
		where $(v_{1} ,v_{2}),(w_{1} ,w_{2}))\in T_{p_{1} } M_{1} \oplus T_{p_{2} } M_{2}$ as before.
	\end{defi}
	\begin{defi}
		Let $M$ be a Riemannian manifold and $X \in \mathfrak{X}(M)$.Let $p \in M$ and let $U\subset M$ be a neiborhood of $p$.Let $\varphi :(-\varepsilon ,\varepsilon )\times U\longrightarrow M$ be a differentiable mapping such that for any $q\in U$ the curve $t\in \varphi (t,q)$ is a trajectory of $X$ passing through $q$ at $t=0$. $X$ is called a Killing field ,if for each $t_{0} \in (-\varepsilon ,\varepsilon )$,the mapping $\varphi(t_{0}):U\subset M:\longrightarrow M$ is an isometry.
	\end{defi}
	Let $\Sigma $ be a oriented Riemannian $m$-manifold and let $F:\Sigma \longrightarrow M^{m+1}$ be an isometric immersion of $Sigma $ into an orientable  Riemannian $(m+1)$-manifold $ M^{m+1}$.We choose a normal unit vector field $N$ along $\Sigma $ and define the shape operator $A$ associated with the second fundamental form of  $\Sigma $;that is ,for any $p\in \Sigma $,
	$$<A(X),Y>=-<\bar{\bigtriangledown  } _{X} N,Y>,X,Y\in T_{p} \Sigma,  $$
	where $\bar{\bigtriangledown  }$ is the Riemannian connection of $\mathbb{M} ^{m+1}$.\
	Associated with  the second fundamental form $A$ there are $n$ algebraic invariants, which are the elementary symmetric functions $Sr$ of its principal curvatures $k_{1},..., k_{n}$,given by $$S_{r} =S_{r} (k_{1},..., k_{n})=\sum_{i_{1}< \cdots <i_{r}}^{} k_{i_{1}},..., k_{i_{r}}\ \ \ \ \ \ 1\le r\le n$$
	As it is well known,the $r$-mean curvature $H_{r}$ of the hypersurface $\Sigma ^{n}$ is defined by $$\begin{pmatrix}
		n\\
		r
	\end{pmatrix}H_{r} =S_{r} (k_{i_{1}},..., k_{i_{r}})  .$$
	i.e.$$H_{r} =\frac{S_{r} }{C_{n}^{r} } ,$$
	where $$C_{n}^{r}= \frac{m!}{r!(m-r)!}.$$
	In particular ,when $r=1$,$$ H_{1}=\frac{1}{n} \sum_{i}k_{i} =\frac{1}{n} tr(A)=H$$
	is just the mean curvature of $\Sigma ^{n}$,which is the main extrinsic curvature of the hypersurface.When $r=2$,$H_{2}$ defines a geometric quantity which is related to the scalar curvature $S$ of the hypersurface.If $ M^{m+1}=R\times _{\rho } M^{m}$, the constant $k$-mean curvature denotes $$\mathcal{H}(t)=(\frac{\rho ^{'}(t) }{\rho (t)})^{k}. $$\\
	For each $0\le r\le n$,one defines the $r$-Newton transformation $P_{r} =\mathfrak{X} (\Sigma )\longrightarrow \mathfrak{X} (\Sigma )$ of the hypersurface $\Sigma ^{n} $ by setting $P_{0}=I$ and,for $ 1\le r\le n$,$$P_{r}=\begin{pmatrix}
		n \\
		r
	\end{pmatrix}H_{r}I-AP_{r-1}.$$
	
	i.e.$$P_{r}=\sum_{j=0}^{r} \begin{pmatrix}
		n \\
		j
	\end{pmatrix}(-1)^{r-j} H_{j} A^{r-j} .$$
	Furthermore,it is easy to see that $P_{r}$ is a self-adjoint operator which commutes with the second fundamental form $A$,that is ,if a local orthonormal frame on $\Sigma ^{n}$ diagonalizes $A$,then it also diagonalizes each $P_{r}$.More specifically,if ${E_{1},...,E_{n}} $ is such a local orthonormal frame with $A(E_{i} )=k_{i} E_{i} $,then $$P_{r} (E_{i} )=\mu _{i,r} E_{i},$$
	where $$\mu _{i,r} =\sum_{i_{1}< \cdots <i_{r},i_{j\ne i}}^{} k_{i_{1}},..., k_{i_{r}}.$$
	It follows from here that for each $0\le r\le n-1$,we have$$tr P_{r} =c_{r} H_{r} c_{r}=(n-r) \begin{pmatrix}
		n \\
		r
	\end{pmatrix}=(r+1)\begin{pmatrix}
		n \\
		r+1
	\end{pmatrix}.$$
	
	Let us consider a domain $D\subset \Sigma $ such that its closure $\bar{D} $ is compact with smooth boundary.\
	\begin{defi}
		A variation of $D$ is differentiable map $\phi:(-\varepsilon ,\varepsilon)\times \Sigma \longrightarrow  M^{m+1} $,where $\varepsilon>0$, such that for each $s\in (-\varepsilon ,\varepsilon)$ the map $\phi _{s} :\Sigma \longrightarrow M^{m+1}$ defined by $\phi _{s} (p)=\phi(s,p)$ is an immersion and $\phi _{0} (p)=F(p)$ for every $p\in \Sigma$ (we recal that $F$ denotes the immersion of $\Sigma$ in $M^{m+1}$) and $\phi _{s} (p)=F(p)$ for $p\in \Sigma \setminus \bar{D}$ and $s\in (-\varepsilon ,\varepsilon)$.\
		We $E_{s} (p)=\frac{\partial \phi  }{\partial s}(s,p)  $ and $f_{s}=<E_{s},N_{s}>$,\
		where $N_{s}$ is the unit normal vector field along $\phi_{s} (\Sigma )$.$E$ is called the variational vector field of $\phi$.Let $A_{s} (p)$ be the shape operator of $\phi_{s} (\Sigma )$ at the point $p$ and $S_{r} (s,p)$ the $r$-th symmetric function of the eigenvalues of  $A_{s} (p)$.
	\end{defi} 
	\begin{defi}
		Let $g\in \mathcal{C } ^{2}(\Sigma )$.We define $L_{r}(g)=div(T_{r} \bigtriangledown g),0\le r\le m.$
	\end{defi}
	M.F. Elbert proved that ,for $1\le r\le m.$\
	$$\frac{\partial S_{r} }{\partial s} =L_{r-1}(f_{s} )+ f_{s} (S_{1}S_{r}-(r+1)S_{r+1})+ f_{s}tr(T_{r-1}\bar{R}_{N})+E_{s}^{T}(S_{r}),$$
	where $\bar{R}_{N}$ is defined as $\bar{R}_{N}(X)=\bar{R}(N,X)N$,$\bar{R}$ is the curvature tensor of $M^{m+1}$ and $E_{s}^{T}$ denotes the tangent part of $E_{s}$.\
	
	\begin{defi}
		Let $\Sigma $ be a connected orientable hypersurface immersed in $\mathbb{R} \times_{\rho } M^{2}$. The height function ,denoted by $h$,of $\Sigma $ in $\mathbb{R} \times_{\rho } M^{2}$ is defined as the restriction to $\Sigma $ of the projection $t:\mathbb{R} \times_{\rho } M^{2}\longrightarrow \mathbb{R}$,and $T=\frac{\partial}{\partial t} $ is the Killing field.
	\end{defi}
	Let $\Sigma $ be an oriented connected hypersurface immersed in $\mathbb{R} \times_{\rho }  M^{2}$ and $f\in \mathcal{C} ^{1} ([0,\infty ))$.Let us suppose that the first and second mean curvature $H_{1} (s)$,$H_{2} (s)$ of $\phi _{s} (\Sigma )$ satisfy $$H_{1}-f(H_{1}^{2} -H_{2}) =0$$.
	The first variation of the member of this identity at $s=0$ gives us $$((1-2H_{1}f^{'} (H_{1}^{2} -H_{2}))\frac{\partial H_{1} }{\partial s} +f^{'}(H_{1}^{2})\frac{\partial H_{2} }{\partial s })_{|_{s=0} } =0.$$
	The principal parts of $\partial _{s} H_{1}(0)=\frac{1}{m} \partial _{s} S_{1}(0) $ and $\partial _{s} H_{2}(0)=\frac{2}{m(m-1)} \partial _{s} S_{2}(0) $ are respectively $\frac{L_{0} }{m}$ and $\frac{2}{m(m-1)}L_{1} $.\\
	when $m=2$ the linearized operator of $H_{1}-f(H_{1}^{2} -H_{2}) =0$ reduces to $$L_{f}=(\frac{1-2ff^{'} }{2}\bigtriangleup + f^{'}L_{1}).$$
	
	To show the main therorem we need the following lemmas \cite{Al10,Al11,Al15} .
	\begin{lem}\label{L2.6}
		If the function $f$ is elliptic ,that is $4x(f^{'} (x))^{2} <1$ for all $x\ge 0$,then the eigenvalues of the operator $L_{f} $ are positive,i.e. $L_{f} $ is elliptic.
	\end{lem}
	\begin{lem}
		Let $\Sigma _{1} ,\Sigma _{2} $ be two special Weingarten surface in $R\times \rho M^{2}$ with respect to the same elliptic function $f$ .Let us suppose that:\
		
		(i)$\Sigma _{1} ,\Sigma _{2} $ are tangent at an interior point $p\in \Sigma _{1} \cap \Sigma _{2}$ or\
		
		(ii)there exists $p\in \partial \Sigma _{1} \cap \partial \Sigma _{1}$ such that both $T_{p} \Sigma _{1} =T_{p} \Sigma _{2}$ and $T_{p} \partial \Sigma _{1}=T_{p} \partial \Sigma _{2}$.\
		
		Also suppose that the unit normal vectors of $\Sigma _{1} ,\Sigma _{2} $ coincide at $p$.If $\Sigma _{1}\ge \Sigma _{2} $ in a neighbourhood $U$ of $p$,then $\Sigma _{1} =\Sigma _{2}$ in $U$.In the case $\Sigma _{1} ,\Sigma _{2} $ have no boundary,$\Sigma _{1} =\Sigma _{2}$ .
	\end{lem} 
	
	\begin{lem}\label{L2.5}
		If $H _{1}(s)$,and $H _{2}(s)$ verify $H _{1}-f(H _{1}^2-H _{2})=0$ for each $s$,then $$(\frac{2}{1-2ff^{'} } )2<\bigtriangledown H_{1} ,\bigtriangledown h>+f^{'}<\bigtriangledown H_{2},\bigtriangledown h>=0$$
	\end{lem}
	
	\begin{lem}\label{L2.7}
		Let $f:\Sigma \longrightarrow R\times \rho M^{n}$ be a compact $k$-mean curvature hypersurface,$1\le k\le n$,with boundary $f(\partial \Sigma )\subset M_{\tau } $for some $\tau \in R$.Then the following holds:\\
		(i)If $H_{k}\le inf_{[\tau,+\infty)} \mathcal{H_{k} } $and $H_{\tau} >0, H^{'} \ge 0$ on $[\tau,+\infty )$ for $k\ge 2$,then $h\le \tau$;\\
		(ii)$H_{k}\ge sup_{-\infty,\tau}\mathcal{H_{k} }$ and either $H_{2}> 0$ or there exists an elliptic point of $\Sigma$ when $k\ge 3$ ,then $h\ge \tau.$
	\end{lem}
	\begin{lem}\label{L2.8}
		Let $\varphi : \Sigma^{n} \longrightarrow \bar{M^{n+1}  }$be a two-sided hypersurface immersed into a Riemannian warped product $\bar{M^{n+1}  }=R\times \rho M^{n}$.For every $r=0,1,...,n-1:$\\
		(a)The height function $h$ satisfies
		$$L_{r}h=(log\rho )^{'}(h)(c_{r} H_{r} -<P_{r}\bigtriangledown h, \bigtriangledown h>)+c_{r}\Theta H_{r+1}$$  
		
		(b)Let $\sigma (t)=\int_{t_{0} }^{t} \rho (r)dr$.Then
		$$L_{r} \sigma (h)=c_{r} (\rho^{'}(h)H_{r} +\rho(h)\Theta H_{r+1})$$,
		
		where $c_{r}:=(n-r)\begin{pmatrix} n\\
			r
		\end{pmatrix} 
		=(r+1)\begin{pmatrix}
			n\\
			r+1
		\end{pmatrix}.$
		
		(c)Let $\tilde{\Theta } =\rho\Theta$.Then,
		\begin{equation*}
			\begin{split}
				L_{r}\tilde{\Theta }=-\frac{c_{r}\rho(h)}{r+1} <\bigtriangledown H_{r+1} ,\bigtriangledown h>-c_{r}\rho^{'}(h)H_{r+1}\\
				-\frac{c_{r}\rho(h)\Theta }{r+1}(nHH_{r+1}-(n-r-1)H_{r+2})\\
				-\frac{\tilde{\Theta } }{\rho^{2}(h)}\sum_{i=1}^{n}\mu i,kK_{M}((N^{*},E_{i}^{*})\left |N^{*}\wedge E_{i}^{*} \right |^{2}\\
				-\tilde{\Theta } (\log  \rho)^{''}(c_{r} H_{r} \left | \bigtriangledown h \right |^{2}-<P_{r} \bigtriangledown h,\bigtriangledown h>)    
			\end{split}
		\end{equation*}
		where ${E_{1},E_{2},...,E_{n}}$ is an orthonormal frame on $\Sigma^{n}$ diagonalizing $A$,$K_{M}$ denotes the Guass curvature of the fiber $M^{n}$,$\mu_{i,r}$ stands for the eigenvalues of $P_{k}$ and,for every vector field $X\in \mathfrak{X} (\bar{M})$,$X^{*}$ is orthogonal projection on $TM$.  
	\end{lem}
	
	\section{Height estimate for special Weingarten surface}\label{S3}
	
	In recent years,the research about mean curvature has gradually become popular,especially the study of the height of surfaces or hypersurfaces with constant mean curvature,which was initially proposed by Heiz in 1969\cite{Al1} ,and the conclusion was optimized and improved by Hoffman in 2003\cite{Al5},Cheng,X et al.in 2005\cite{Al9},Alias,L.J et al. in 2007\cite{Al9},and de lima,E,L et al. in 2018\cite{Al11},respectively.\
	
	On the other hand,in 1994,Rosenberg,H et al. \cite{Al14} studied the height estimation of,the SWET surface,which differs from the previous study in that the SWET surface does not have a mean curvature of $0$ alone,but rather the mean curvature and Gaussian curvature of the surface satisfy a certain relationship.And Rosenberg proved has proved the SWET surface has height in Euclidean space. In 2014,Filippo morabito \cite{Al15} extended the Rosenberg’s conclusion to the product space.\
	
	The main research of this paper is to prove that the SWET surface also has the corresponding height conclusion in the warped product space.\
	
	\begin{lem}\label{L3.1}
		$N^{*} \wedge E_{i} ^{*} =\left | \bigtriangledown h \right |^{2} -<E_{i} ,\bigtriangledown h >^{2}$
		
	\end{lem}
	\begin{proof}
		
		Firstly,we compute	$$|N^{*} \wedge E_{i} ^{*}| =1-<E_{i},T>^{2}-<N,T>^{2}.$$\
		
		Since $$N^{*}=N-<N,T>T,E_{i} ^{*}=E_{i}-<E_{i},T>T,$$\
		
		and $$N^{*} \wedge E_{i} ^{*}=<N^{*},N^{*}><E_{i}^{*},E_{i}^{*}>-<E_{i}^{*},N^{*}>^{2}.$$
		
		In addition to,	
		\begin{equation*}
			\begin{aligned}
				<N^{*},N^{*}>&=<N-<N,T>T,N-<N,T>T>\\
				&=<N,N>-<N,T><N,T>-<N,T><N,T>\\
				&+<N,T><N,T><T,T>\\
				&=1-2<N,T><N,T>+<N,T><N,T>\\
				&=1-<N,T><N,T>\\
			\end{aligned}
		\end{equation*}
		\begin{equation*}
			\begin{aligned}
				<E_{i}^{*},E_{i}^{*}>&=<E_{i}-<E_{i},T>T,E_{i}-<E_{i},T>T>\\
				&=<E_{i},E_{i}>-<E_{i},T><T,E_{i}>\\
				&+<E_{i},T><E_{i},T><T,T>\\
				&=1-2<E_{i},T><T,E_{i}>+<E_{i},T><T,E_{i}>\\
				&=1-<E_{i},T><T,E_{i}>\\
			\end{aligned}.
		\end{equation*}
		\begin{equation*}
			\begin{aligned}
				<E_{i}^{*},N^{*}>&=<E_{i}-<E_{i},T>T,N-<N,T>T>\\
				&=<E_{i},N>-<E_{i},T><N,T>\\
				&-<E_{i},T><T,N>+<E_{i},T><N,T><T,T>\\
				&=<E_{i},N>-2<E_{i},T><N,T>+<E_{i},T><N,T>\\
				&=<E_{i},N>-<E_{i},T><N,T>\\	
			\end{aligned}
		\end{equation*}
		\begin{equation*}
			\begin{aligned}
				<E_{i}^{*},N^{*}>^{2}&=<<E_{i},N>-<E_{i},T><N,T>,<E_{i},N>\\
				&-<E_{i},T><N,T>>\\
				&=<E_{i},N><E_{i},N>-<E_{i},N><E_{i},T><N,T>\\
				&-<E_{i},T><N,T><E_{i},N>\\
				&+<E_{i},T><N,T><E_{i},T><N,T>\\
			\end{aligned}	
		\end{equation*}
		
		Since,	$<E_{i},N>=0$,and then  $<E_{i}^{*},N^{*}>^{2}=<E_{i},T>^{2}<N,T>^{2} $ .\
		
		We can get,
		\begin{equation*}
			\begin{aligned}
				N^{*} \wedge E_{i} ^{*}&=<N^{*},N^{*}><E_{i}^{*},E_{i}^{*}>\\
				&=-<E_{i}^{*},N^{*}>^{2}\\
				&=(1-<N,T><N,T>)(1-<E_{i},T><T,E_{i}>)\\
				&=-<E_{i},T>^{2}<N,T>^{2} \\
				&=(1-<N,T>^{2})(1-<E_{i},T>^{2})\\
				&=-<E_{i},T>^{2}<N,T>^{2}\\
				&=1-<E_{i},T>^{2}-<N,T>^{2}+\\
				&=<N,T>^{2}<E_{i},T>^{2}\\
				&=-<E_{i},T>^{2}<N,T>^{2}\\
				&=1-<E_{i},T>^{2}-<N,T>^{2}\\
			\end{aligned}
		\end{equation*}
		
		Secondly,we compute $\left | \bigtriangledown h \right |^{2} -<E_{i} ,\bigtriangledown h >^{2}$.\
		
		Since $$\bigtriangledown h=T-<N,T>N,$$
		therefore,
		\begin{equation*}
			\begin{aligned}
				|\bigtriangledown h|^{2} &=<T-<N,T>N,T-<N,T>N>\\
				&=<T,T>-<T,N><N,T>\\
				&=-<N,T><N,T>+<N,T><N,T><N,N>\\
				&=1-2<T,N><N,T>+<N,T><N,T>\\
				&=1-<T,N>^{2},\\
			\end{aligned}
		\end{equation*}
		and
		\begin{equation*}
			\begin{aligned}
				<E_{i} ,\bigtriangledown h>&=<E_{i},T-<N,T>N>\\
				&=<E_{i},T>-<E_{i},N><N,T>\\
				&=<E_{i},T>,\\
			\end{aligned}
		\end{equation*}
		So,$$<E_{i} ,\bigtriangledown h>^{2}=<E_{i},T>^{2}.$$
		
		and then $$\bigtriangledown h^{2}-<E_{i} ,\bigtriangledown h>^{2}=1-<T,N>^{2}-<E_{i},T>^{2}.$$\
		
		Finally,we get
		\begin{equation*}
			\begin{aligned}
				N^{*} \wedge E_{i} ^{*}&=1-<T,N>^{2}-<E_{i},T>^{2}\\
				&=\bigtriangledown h^{2}-<E_{i} ,\bigtriangledown h>^{2}
			\end{aligned}
		\end{equation*}
		i.e.$$N^{*} \wedge E_{i} ^{*}=\bigtriangledown h^{2}-<E_{i} ,\bigtriangledown h>^{2}.$$\
		We finish this proof.
		
	\end{proof}
	
	\begin{lem}\label{L3.2}
		$\sum_{i=1}^{2}  \mu _{i,k} =c_{k} H_{k} ,
		\sum_{i=1}^{2} \mu _{i,k}<E_{i} ,\bigtriangledown h>=<P_{k} \bigtriangledown h,\bigtriangledown h>$.
	\end{lem}
	\begin{proof}
		Firstly,we proof $$\sum_{i=1}^{2}  \mu _{i,k} =c_{k} H_{k} .$$
		Since,$$P_{r} (E_{i} )=\mu _{i,k}E_{i},$$\\
		and $$\mu _{i,k}=\sum_{i_{1}<i_{2}<\cdots <i_{r},i_{j\ne}  }^{} k_{i_{1} }\cdots k_{i_{r} }.$$\\
		Therefore,$$tr(P_{r})=\sum_{i=1}^{2}\mu _{i,r},$$\\
		moreover,$$H_{r} =\frac{S_{r} }{c_{r} } =\frac{ \sum_{1}^{2}\mu _{i,r} }{c_{r} },$$\\
		So $$tr(P_{r})=c_{r}H_{r}.$$
		
		Next,we prove $\sum_{i=1}^{2} \mu _{i,k}<E_{i} ,\bigtriangledown h>=<P_{k} \bigtriangledown h,\bigtriangledown h>$.\\
		Since,$$P_{r}=\sum_{i=1}^{2}\mu _{i,r}E_{i}. $$\\
		We can get.
		\begin{equation*}
			\begin{aligned}
				<E_{i},\bigtriangledown h><E_{i},\bigtriangledown h>
				&=\sum_{i=1}^{2}< \mu _{i,r}E_{i} ,\bigtriangledown h><E_{i},\bigtriangledown h>\\
				&=<P_{k}\bigtriangledown h,\bigtriangledown h> 
			\end{aligned}
		\end{equation*}
	\end{proof}
	\begin{thm}\label{U3.3 }
		let $\Sigma$ be a compact SWET surface,wuich  embedded in  $R\times \rho M^{2}$,and $ \rho$ is a positive increasing function,$\partial \Sigma \subset M^{2} \times \left \{ t_{0}  \right \},t_{0} \in \mathbb{R}$.Suppose the Guassian curvature of $M$ is $K_{M}$, $K_{M}\ge Sup_{R} ((\rho^{'})^{2}-\rho^{''}\rho)$,$ H_{k} \ge Sup_{(-\infty ,t_{0} ]} \mathcal{H} _{k},k=1,2. $ if $f^{'}> 0,1-2ff^{'}>0,f^{2} +x(1-4ff^{'})>0,f-2f^{'}x>0.$then,
		
		\[h\le t_{0}+\frac{\rho(maxh)}{c\rho(t_{0})},c=\min_{\Sigma } (\frac{ \rho^{'}}{\rho} +\frac{f^{2} +x(1-4ff^{'}}{f-2f^{'}x} ).
		\]
	\end{thm}
	\begin{proof}
		By lemma 2.10,we can get $h\ge t_{0}$.Because $\Sigma $ is caompact,there is a point $p_{0}$,which makes $h$ reach the maximum,i.e.$\bigtriangledown h(p_{0})=0,\Theta (p_{0})=\pm 1$.According to lemma $0\ge \bigtriangleup h(p_{0})=n\mathcal{H}(t_{0} )+n\Theta (p_{0})H_{1} (p_{0})>n\Theta (p_{0})H_{1} (p_{0}).$ so $\Theta (p_{0})=-1.$ And $\Theta$ is a negative  function.
		
		Let $\varphi =c\sigma (h)+\tilde{\Theta }$ ,then  $L_{f} \varphi=L_{f}(c\sigma (h)+\tilde{\Theta })= cL_{f} (\sigma)+L_{f} (\tilde{\Theta } )$
		
		\begin{equation*}
			\begin{split}
				=c[(\frac{1-2ff^{'}}{2}) \cdot 2(\rho^{'}+\rho\Theta H_{1})+2f^{'}(\rho^{'}H_{1}+\rho\Theta H_{2})]\\
				+\frac{1-2ff^{'}}{2}[-2\rho<\bigtriangledown H_{1},\bigtriangledown h>-2\rho^{'} H_{1}-2\rho\Theta(2H_{1} ^{2}-H_{2})]-A\\ 
				+f^{'}[-\rho^{'}<\bigtriangledown H_{2},\bigtriangledown h>-2\rho^{'}H_{2}-2\rho\Theta H_{1}H_{2})-B\\
				=c[(1-2ff^{'})(\rho^{'}+\rho\Theta H_{1})+2f^{'}\rho^{'}H_{1}+2f^{'}\rho\Theta H_{2}]\\
				+(1-2ff^{'})[-\rho^{'}H_{1}-\rho\Theta (2H_{1} ^{2}-H_{2})]+f^{'}(-2\rho^{'}H_{2}-2\rho\Theta H_{1}H_{2})-A-B 
			\end{split}
		\end{equation*}
		
		and,\begin{equation*}
			\begin{split}
				A=\frac{1-2ff^{'}}{2}[\frac{\tilde{\Theta } }{\rho^{2}}  \sum_{i=1}^{2} \mu _{i,0}K_{M}(N^{*},E_{i}^{*})\left |N^{*}\wedge E_{i}^{*} \right |^{2}\\
				+\tilde{\Theta } (\log  \rho)^{''}(c_{0} \left | \bigtriangledown h \right |^{2}-<P_{0} \bigtriangledown h,\bigtriangledown h>)].
			\end{split}
		\end{equation*}
		\begin{equation*}
			\begin{split}
				B= f^{'}[\frac{\tilde{\Theta } }{\rho^{2}}  \sum_{i=1}^{2} \mu _{i,1}K_{M}(N^{*},E_{i}^{*})\left |N^{*}\wedge E_{i}^{*} \right |^{2}\\
				+\tilde{\Theta } (\log  \rho)^{''}(c_{1} \left | \bigtriangledown h \right |^{2}-<P_{1} \bigtriangledown h,\bigtriangledown h>)].
			\end{split}
		\end{equation*}
		
		Let $x=H_{1} ^{2}-H_{2},$
		\begin{equation*}
			\begin{split}
				L_{f} \varphi=c[(1-2ff^{'})(\rho^{'}+\rho\Theta f)+2ff^{'}\rho^{'}+2f^{'}\rho\Theta(f^{2}- x)]\\
				+(1-2ff^{'})[-\rho^{'}f-\rho\Theta (f^{2}+x)]+f^{'}[-2\rho^{'}(f^{2}-x)-2\rho\Theta f(f^{2}-x)]-A-B\\
				=c(\rho^{'}+\rho\Theta f-2ff^{'}\rho^{'}-2f^{2}f^{'}\rho\Theta+2ff^{'}\rho^{'}+2f^{2}f^{'}\rho\Theta-2f^{'}\rho\Theta x)\\
				-\rho^{'}f-\rho\Theta f^{2}-\rho\Theta x+2f^{2}f^{'}\rho^{'}+2f^{3}f^{'}\rho\Theta +2ff^{'}\rho\Theta x\\
				-2\rho^{'}f^{2}f^{'}+2f^{'}\rho^{'}x-2\rho\Theta f^{3}f^{'}+2\rho\Theta ff^{'}x-A-B\\
				=c(\rho^{'}+\rho\Theta (f-2f^{'}x))-\rho^{'}(f-2f^{'}x)-\rho\Theta(f^{2}+x(1-4ff^{'}))-A-B.   
			\end{split}
		\end{equation*}
		
		Next,we prove $-A-B\ge 0$
		
		Let $$\alpha =Sup_{R} ((\rho^{'})^{2}-\rho^{'}\rho)$$
		since $P_{1}$ is positive definite ,\\
		then $$\mu _{i,k}>0,k=1,2$$
		and then
		\[\mu _{i,k}K_{M}(N^{*},E_{i}^{*})\left |N^{*}\wedge E_{i}^{*} \right |^{2}\ge \alpha \mu _{i,k}\left |N^{*}\wedge E_{i}^{*} \right |^{2}.\]
		
		But,$$\left |N^{*}\wedge E_{i}^{*} \right |^{2}=\left | \bigtriangledown h \right | ^{2} -<E_{i} ,\bigtriangledown h> ^{2}.$$
		
		so,
		\begin{equation*}
			\begin{split}
				\sum_{i=1}^{2} \mu _{i,k}K_{M}(N^{*},E_{i}^{*})\left |N^{*}\wedge E_{i}^{*} \right |^{2}\ge \sum_{i=1}^{2}\alpha \mu _{i,k}(\left | \bigtriangledown h \right | ^{2} -<E_{i} ,\bigtriangledown h> ^{2} )\\
				=\alpha (c_{k} H_{k} \left | \bigtriangledown h \right | ^{2}-<P_{k} \bigtriangledown h,\bigtriangledown h>)\ge 0. 
			\end{split}
		\end{equation*}
		
		Therefore we can get:
		\begin{equation*}
			\begin{split}
				\frac{1 } {\rho^{2}}  \sum_{i=1}^{2} \mu _{i,k}K_{M}(N^{*},E_{i}^{*})\left |N^{*}\wedge E_{i}^{*} \right |^{2}+ (\log  \rho)^{''}\\
				(c_{k} H_{k} \left | \bigtriangledown h \right |^{2}-<P_{k} \bigtriangledown h,\bigtriangledown h>)\ge 0.    
			\end{split}
		\end{equation*}
		
		According to the conditons in the theorem 
		$$1-2ff^{'}>0,f^{'}>0,$$and $\tilde{\Theta } <0$
		
		finally we have $$-A-B\ge 0.$$
		
		So ,if we want to $$L_{f} \varphi \ge 0$$,we only prove
		\[c(\rho^{'}+\rho\Theta (f-2f^{'}x))-\rho^{'}(f-2f^{'}x)-\rho\Theta(f^{2}+x(1-4ff^{'}))\ge 0.\]
		
		It is clear from the assumption \[\rho^{'}\ge 0,f-2f^{'}x> 0.\]
		
		so,
		\begin{equation*}
			\begin{aligned}
				c&=\le \frac{\rho^{'}(f-2f^{'}x)+\rho\Theta(f^{2}+x(1-4ff^{'}))}{\rho\Theta (f-2f^{'}x)} \\
				&\le\frac{\rho^{'}}{\rho\Theta}+\frac{f^{2}+x(1-4ff^{'})}{f-2f^{'}x}\\
				&\le\frac{\rho ^{'} }{\rho} +\frac{f^{2}+x(1-4ff^{'})}{f-2f^{'}x}.  
			\end{aligned}
		\end{equation*}
		
		Let $$c=\min_{\Sigma } (\frac{\rho ^{'} }{\rho} +\frac{f^{2}+x(1-4ff^{'})}{f-2f^{'}x}),$$
		
		then $$L_{f} \varphi \ge 0.$$
		
		According to the maxmum principle,
		$$\max_{\Sigma } \varphi=\max_{\partial\Sigma }\varphi.$$
		
		then $$c\sigma (h)- \rho\le c\sigma (t_{0} ),$$
		
		in other words $$c(\sigma (h)-\sigma (t_{0} ))\le \rho.$$
		
		And $\rho$ is a increasing function,so
		$$\forall t\ge t_{0} ,
		\sigma (t)-\sigma (t_{0} )\ge\rho(t_{0})(t-t_{0}).$$
		
		Since $$h\ge t_{0},$$
		
		we have $$\sigma (h)-\sigma (t_{0} )\ge\rho(t_{0})(h-t_{0}),$$
		
		and then
		$$c\rho(t_{0})(h-t_{0})\le\rho.$$
		
		finally we get the height estimate of SWET surface 
		$$h\le t_{0}+\frac{\rho (maxh)}{c\rho(t_{0})} .$$
		
	\end{proof}
	
	\vspace{2.12mm}
	\subsection*{Acknowledgment}
	I am grateful to Professor Hengyu  (Chongqing University) for many useful discussions and suggestions.

\end{document}